\documentclass[a4paper,11pt]{article}         
\usepackage{graphicx}
\usepackage{lmodern}
\usepackage{amsmath}  
\usepackage{amsfonts} 
\usepackage{amsthm}
\usepackage{graphicx}
\usepackage{array} 
\usepackage{enumerate} 
\usepackage[latin1]{inputenc} 
\usepackage[T1]{fontenc} 
\usepackage{url}
\usepackage[usenames]{xcolor}
\usepackage{framed}\colorlet{shadecolor}{black!05!white} 
\usepackage{fancyhdr}
\fancyheadoffset{30pt} 
\fancyfootoffset{30pt} 

\newcommand{\Reals}{\mathbb{R}}

\newcommand{\E}{\operatorname{E}}
\newcommand{\ind}{\mathbf{1}}
\newcommand{\rlim}{{\scriptstyle+}}
\newtheorem{theorem}{Theorem}[section]

\newtheorem{lemma}[theorem]{Lemma}
\newtheorem{remark}[theorem]{Remark}
\newtheorem{remarks}[theorem]{Remarks}
\newtheorem{corollary}[theorem]{Corollary}
\newtheorem{proposition}[theorem]{Proposition}
\newtheorem{example}[theorem]{Example}
\newtheorem{examples}[theorem]{Examples}
\newlength\savedwidth

\begin{document}  
\title{{On Uniqueness for some non-Lipschitz SDE}}
\author{Aureli Alabert
                               \\ 
           Department of Mathematics \\
           Universitat Autònoma de Barcelona \\  
           08193 Bellaterra, Catalonia \\  
           \url{Aureli.Alabert@uab.cat}
           \and
           Jorge A. León
           \\ 
           Department of Automatic Control \\
           Cinvestav-IPN \\  
           Apartado Postal 14-740 \\
           07000 M\'exico D.F, Mexico\\  
           \url{jleon@ctrl.cinvestav.mx}
}

\maketitle

\begin{abstract}  
We study the uniqueness in the path-by-path sense (i.e. $\omega$-by-$\omega$) of 
solutions to stochastic differential equations with additive noise and non-Lipschitz autonomous
drift. The notion of path-by-path solution involves considering a collection of ordinary differential equations
and is, in principle, weaker than that of a strong solution, since no adaptability condition is required.
We use results and ideas from the classical theory of ode's, together with probabilistic tools like
Girsanov's theorem, to establish the uniqueness property for some classes of noises, including
Brownian motion, and some drift functions not necessarily bounded nor continuous.

\end{abstract}  
  
\section{Introduction}

Consider the stochastic differential equation (sde)
\begin{equation}\label{eq-x0bW} 
  X_t=x_0+\int_0^t b(X_s)\,ds + W_t\ , \quad t\in[0,T]\ ,
\end{equation}
where $W$ is some noise process with continuous paths. That means, $W$ is a random variable defined
on some complete
probability space $(\Omega,{\cal F}, P)$ with values in the space $C({[0,T]})$ of real continuous functions
on $[0,T]$, endowed with its Borel $\sigma$-field. A canonical example is Brownian motion. The function
$b\colon\Reals\rightarrow\Reals$ is supposed to be measurable at least, and
$x_0$ is a given real number. We refer the reader to Karatzas and Shreve \cite{MR1121940} for the
concepts on stochastic processes that we use in this paper.

We recall that a \emph{strong solution} of the equation above is a 
stochastic process $X$, with measurable paths, adapted to the filtration 
generated by $W$, 
and such that, for every $t$, the random variable $X_t-x_0-\int_0^t b(X_s)\,ds$ is well defined 
and is equal to $W_t$ almost surely.
In fact, the form of the equation implies that $X$ must have also continuous paths, hence 
the processes $\{X_t-x_0-\int_0^t b(X_s)\,ds,\ t\in[0,T]\}$ and 
$\{W_t,\ t\in [0,T]\}$ will be indistinguishable, i.e.\ they will be equal as $C([0,T])$-valued random variables.
It makes sense also to speak about local solutions, where the process $X$ exists only up to some
(random) time $\tau$.

\emph{Uniqueness} of solutions for the sde's  (\ref{eq-x0bW}) (sometimes called \emph{strong uniqueness} or \emph{pathwise uniqueness})
means that given a probability space, a process with the law of $W$ defined on it, and the initial condition $x_0$,
two strong solutions are indistinguishable.
The classical existence and uniqueness result for sde of the type (\ref{eq-x0bW}) is
the following (see, e.g.\ \cite[Theorems 5.2.5 and 5.2.9]{MR1121940}): 
\begin{theorem}
  If $b$ is a Lipschitz function, 
  then there is a unique  strong solution to (\ref{eq-x0bW}).
  \end{theorem}

Existence and uniqueness of a strong solution can be proved under much weaker conditions on $b$, 
at least for the case of a Brownian motion $W$. Indeed, it was shown by Veretennikov \cite{MR568986} that it is enough 
that $b$ be bounded and measurable, also under some non-additive noises.  
This type of result was extended to parabolic differential equations in one space dimension 
driven by a space-time white noise by Bally, Gyöngy and Pardoux \cite{MR1266318}, Gyöngy \cite{MR1608641}
and Alabert and Gyöngy \cite{MR1811748}. In the latter, as well as in Gyöngy and Martínez \cite{MR1864041} in $\Reals^d$, 
the drift $b$ is allowed to be locally unbounded, provided a suitable integrability condition holds. 
We refer the reader to Flandoli \cite[Chapter 2]{MR2796837} for a more complete discussion on the topic.
For processes other than 
Brownian Motion, we can mention Nualart and Ouknine \cite{MR1934157}, \cite{MR2069689}. 
We cite also Catellier and Gubinelli \cite{arXiv:1205.1735v1},  where a slightly different problem is considered:
The coefficient $b$ is generalized to non-functions, that means, to distributional fields, leading to 
delicate problems about the meaning of the composition $b(X)$ and the definition of solution itself.

Now we introduce an ordinary differential equation similar to (\ref{eq-x0bW}): 
Given a real continuous function $\omega\in C([0,T])$, we may write
\begin{equation}\label{eq-x0bw}
  x_t=x_0+\int_0^t b(x_s)\,ds + \omega_t\ ,\quad t\in[0,T]\ .
\end{equation}

If $b$ is a function for which the existence of a strong solution of the related sde (\ref{eq-x0bW}) 
has been stated, one can say
immediately that there exists a solution to (\ref{eq-x0bw}) for almost all continuous functions $\omega$ with respect
to the law of $W$. Nothing can be said of any particular $\omega$, however. 

Assume, on the other hand, that we could prove the existence of a solution to (\ref{eq-x0bw}) for a certain class 
of functions $\omega$ having probability one with respect to the law of $W$. Would this yield an existence theorem for the sde?
This is not clear, since the condition of adaptability in the definition of strong solution need not be satisfied, in principle.

According to Flandoli \cite{MR2796837}, we will call \emph{path-by-path solution} of the sde (\ref{eq-x0bW}) a solution
obtained by solving $\omega$-by-$\omega$ the corresponding class of ode's (\ref{eq-x0bw}). Existence of a strong solution implies existence of a 
path-by-path solution,
but the converse is not known to be true, in general.     
Similarly, uniqueness of the
path-by-path solution does imply uniqueness in the strong sense, but not the other way round.

We ask ourselves if this gap can always be closed or, on the contrary, if it is possible to find counterexamples. This seems to
be a difficult problem. 
Notice that in the classical case ($b$ Lipschitz), it is true that a path-by-path solution is also strong. This
is due to the Picard iteration scheme, which implies the existence 
of a strong solution, and to Gronwall's lemma, which gives the uniqueness in the path-by-path sense. 
Therefore, the question concerns only the
non-Lipschitz cases.

We insist in the fact that establishing existence and uniqueness for a fixed particular $\omega\in C([0,T])$ is 
a different problem.  
For example, if 
$b(x)=\sqrt{|x|}$, $x_0=0$ and 
$\omega\equiv0$, it is easy to see that the equation has exactly two local solutions (infinitely many, in a global sense), 
namely $x\equiv 0$ and $x_t=t^2/4$.
However, the corresponding stochastic equation with a Brownian Motion $W$ has a unique strong solution (see, e.g.\ \cite{MR1864041}).
Our results in Section \ref{sec-gen-noise} show in particular that the solution is unique also in the path-by-path sense.

Concerning uniqueness of path-by-path solutions, we only know the works of Davie \cite{MR2377011, MR2789085} 
and the remarks on them made by Flandoli \cite{MR2836529}. In \cite{MR2377011} 
it is proved, by means of an estimate quite complicated to obtain, that for a bounded measurable function $b$ 
there is a unique solution to (\ref{eq-x0bw}), for a class of continuous functions
$\omega$ which has probability one with respect to the law of Brownian Motion. Hence, the solution to the corresponding
sde, which was already known to exist in the strong sense, is not only strongly unique, but also path-by-path unique. In \cite{MR2789085} a diffusion
coefficient is introduced, and the equation interpreted in the rough path sense. 
We provide a simpler proof of the path-by-path uniqueness in cases where $b$ is not necessarily bounded or continuous;
however, we have to restrict ourselves to dimension one, whereas in \cite{MR2377011, MR2789085} the equations are $d$-dimensional 
and the function $b$ may also depend on time. 

In this paper we apply some ideas from the theory of ordinary
differential equations to study the path-by-path uniqueness of equation (\ref{eq-x0bW}). 
Existence theorems are very general (e.g.\ Peano and Carathéodory theorems, that can be found 
in classical books like Hartman \cite{MR0171038}); 
however, uniqueness (and non-uniqueness)
results are poor and fragmented in comparison,  and particularly scarce for equations of the form (\ref{eq-x0bw}) (see, for instance,
the book by Agarwal and Lakshmikantham \cite{MR1336820}, dedicated to the subject). 

The paper is organized as follows. In Section \ref{sec-gen-noise} we use an extension of Iyanaga's uniqueness theorem (Theorem \ref{the:3.2n}) 
for ode's, and Girsanov's theorem, to establish our main
result: the path-by-path uniqueness of equation (\ref{eq-x0bW}) for a 
Brownian motion $W$. Next, we see that
the hypotheses on $b$ can be relaxed if the noise has a constant sign, leading 
to similar theorems for the absolute value of the Brownian motion $|W_t|$ and for $-|W_t|$. In Section 3 we consider the 
particular case of the square root:  
Using Lakshmikantham's theorem (Lemma \ref{LakshUT}), we obtain a simpler proof when $b(x)=\sqrt{|x|}$ and the noise is non-negative; moreover, 
a discontinuous version of the 
square root exemplifies that continuity is not essential for the techniques of Section \ref{sec-gen-noise} to work.
Finally, in Section 4 we use the idea behind the proof of Peano's uniqueness theorem to deal with some differentiable noises.

\section{Main results}\label{sec-gen-noise}

Let $\omega\colon[0,T]\rightarrow\Reals$ be a fixed continuous function, with $\omega_0=0$,
 and $b\colon\Reals\rightarrow\Reals$ a measurable function. Consider the 
equation

\begin{equation}\label{eq-bw} 
   x_t=\int_0^t b(x_s)\,ds + \omega_t\ , \quad t\in[0,T]\ .
\end{equation}

Taking $y_t:=x_t-\omega_t$ as a new unknown function, (\ref{eq-bw}) is equivalent to
\begin{equation}\label{eq-bw2} 
   y_t=\int_0^t b(y_s+\omega_s)\,ds\ , \quad t\in[0,T]\ .
\end{equation}

Let us assume that this equation have at least one continuous solution 
$y\colon[0,T]\longrightarrow\Reals$, which is true by the Peano 
existence theorem if $b$ is continuous (see, for instance, Lakshmikantham and Leela
\cite[Theorem 1.1.2]{MR0379933}).

We will find some sufficient conditions on $\omega$ 
ensuring the uniqueness of that solution. Towards this end,
consider the following set of hypotheses on function $b$:
\begin{enumerate}[H1.]
  \item
  $b(0)=0$. 
  \item 
  $b$ is non-decreasing on $(0,\infty)$. 
  \item 
  $b$ is continuous on $[0,\infty)$, and of class $C^1$ with $b'$ non-increasing on $(0,\infty)$. 
  \item
  $b(|x|)\le b(-|x|)$.
  \item
  $b$ is non-increasing on $(-\infty,0]$.
\end{enumerate}
Notice that, under these hypotheses,
any solution $y$ of (\ref{eq-bw2}) is non-negative and non-decreasing. 
We will make use of the following two lemmas:

\begin{lemma}\label{lem-des-iyana}
   Assume hypotheses H1-H5 hold true. Then, 
  for any two solutions $y$ and $\bar y$ of equation (\ref{eq-bw2}), such that $y\le \bar y$, and for any continuous
  function $\omega$ with $\omega_0=0$ and not identically zero on  the interval $[0,T]$, we have the inequality
  \begin{align}
    b(\bar y_t+\omega_t)-b(y_t+\omega_t)
    &\le
    (\bar y_t-y_t)
    \cdot
    b'(|y_t+\omega_t|\rlim)
    \label{eq-des1}
    \\
    &\le 
    (\bar y_t-y_t)
    \cdot
    \Big[
    \ind_{\{\omega_t\ge 0\}}
    \cdot
    b'\Big(\Big(\omega_t+\int_0^t\ind_{\{\omega_s>0\}}b(\omega_s)\,ds\Big)\rlim\Big)
    \nonumber
    \\
    &\phantom{\le 
    (\bar y_t-y_t)
    \cdot
    \Big[}
    +
    \ind_{\{\omega_t<0\}}
    \cdot
    b'(|y_t+\omega_t|\rlim)
    \Big]  \ , 
    \quad t\in[0,T] \ ,
    \nonumber
  \end{align}
  where $b'(z\rlim)$ means $\lim_{x\downarrow z} b'(x)$, and $b'(0\rlim)$ may be infinite.
\end{lemma} 
  \emph{Remark:} We write right-limits 
only because $b'$ is not necessarily defined
at zero.

\begin{proof} 
The proof will be divided into several cases. 
  \begin{description}
    \item [{\sc Case 1:}  \textmd{$t$ such that $\omega_t\ge0$}].
 
 By the mean value theorem, and using that $b'$ is non-increasing on the positive axis,
       \begin{equation*}
         b(\bar y_t+\omega_t)-b(y_t+\omega_t)
         \le
         (\bar y_t-y_t)\cdot b'((y_t+\omega_t)\rlim)
         \ ,
       \end{equation*}
      which, together with (\ref{eq-bw2}), implies
       \begin{align*}
         b(\bar y_t+\omega_t)-b(y_t+\omega_t)
         &\le 
         (\bar y_t-y_t)\cdot
         b'\Big(\Big(\omega_t+\int_0^t b(y_s+\omega_s)\,ds\Big)\rlim\Big)
         \\ & \le 
         (\bar y_t-y_t)\cdot
         b'\Big(\Big(\omega_t+\int_0^t \ind_{\{\omega_s>0\}}b(y_s+\omega_s)\,ds\Big)\rlim\Big)
         \\ & \le 
         (\bar y_t-y_t)\cdot
         b'\Big(\Big(\omega_t+\int_0^t \ind_{\{\omega_s>0\}}b(\omega_s)\,ds\Big)\rlim\Big)
         \ .
       \end{align*}
       
    \item [{\sc Case 2:}  \textmd{$t$ such that $-y_t\le\omega_t< 0$}].

   Similar to the case 1, the mean value theorem gives
 \begin{equation*}
         b(\bar y_t+\omega_t)-b(y_t+\omega_t)
         \le
         (\bar y_t-y_t)
         \cdot b'(|y_t+\omega_t|\rlim) 
         \ .
       \end{equation*}
   \item [{\sc Case 3:}  \textmd{$t$ such that $-\frac{y_t+\bar y_t}{2}\le\omega_t< -y_t$}].
   
   We have $\bar y_t+\omega_t\ge-y_t-\omega_t>0$. Therefore,
       \begin{align*}
         b(\bar y_t+\omega_t)-b(y_t+\omega_t)
         &\le
         b(\bar y_t+\omega_t)-b(-y_t-\omega_t)
         \\ &\le
         (\bar y_t+y_t+2\omega_t)
         \cdot 
         b'(|y_t+\omega_t|\rlim)
         \le
         (\bar y_t-y_t)
         \cdot 
         b'(|y_t+\omega_t|\rlim)
         \ ,
       \end{align*}
       where we have used hypothesis H4 in the first inequality. 
   \item [{\sc Case 4:}  \textmd{$t$ such that $-\bar y_t\le\omega_t<-\frac{y_t+\bar y_t}{2}$}].
   
   Now,  $y_t+\omega_t<-\bar y_t-\omega_t\le0$. Using again hypothesis H4, 
       \begin{equation*}
         b(\bar y_t+\omega_t)-b(y_t+\omega_t)
         \le
         b(\bar y_t+\omega_t)-b(-y_t-\omega_t)
         \le
         0
         \ .
       \end{equation*}
   \item [{\sc Case 5:}  \textmd{$t$ such that $\omega_t<-\bar y_t$}].
   
   Here, we have $y_t+\omega_t\le \bar y_t+\omega_t< 0$. Hence, by H5,
       \begin{equation*}
         b(\bar y_t+\omega_t)-b(y_t+\omega_t)
         \le
         0
         \ .
       \end{equation*}
      
  \end{description}
\end{proof}

\begin{lemma}\label{Pachpatte}
Let $f, g, h\colon [0, T] \rightarrow \Reals$ be continuous functions, and 
$k\colon [0, T] \times \Reals \rightarrow \Reals$ measurable.
Assume 
\begin{equation*}
f(t) \le h(t) + \int_0^t k(s,f(s))\,ds 
\quad\text{and}\quad  
g(t) \ge h(t) + \int_0^t k(s,g(s))\, ds
\ ,\quad 
t\in[0,T]
\ , 
\end{equation*}
and that $k$ is non-decreasing in the second variable and 
\begin{equation*}
k(t,\bar y)-k(t, y) \le a(t)(\bar y-y)
\ ,\quad \text{for $y,\bar y\in\Reals$, $y\le\bar y$}   
\ .
\end{equation*}
for some integrable function $a\colon [0,T]\rightarrow \Reals$.

 Then $f(t) \le g(t)$ for all $t\in[0,T]$. 
\end{lemma}
\begin{proof}
  See Pachpatte \cite[Theorem 2.2.5]{MR1487077}.
\end{proof}

\begin{theorem} \label{the:3.2n}
Let $b$ satisfy hypotheses H1-H5 and $y$ be a solution of (\ref{eq-bw2}). Assume that the function
\begin{equation}\label{cond:1}
a(t)=\Big[
    \ind_{\{\omega_t\ge 0\}}
    \cdot
    b'\Big(\Big(\omega_t+\int_0^t\ind_{\{\omega_s>0\}}b(\omega_s)\,ds\Big)\rlim\Big)
    +
    \ind_{\{\omega_t<0\}}
    \cdot
    b'(|y_t+\omega_t|\rlim)
     \Big]
\end{equation}
belongs to $L^1([0,T])$. Let $\bar y$ be another solution, with $y\le \bar y$. Then, $y=\bar y$.
\end{theorem}
\begin{remarks}\label{rem:g}\hspace{1pt}
\begin{enumerate}
  \item 
  If $b$ is continuous, so that maximal and minimal solutions exist, one may say that a solution $y$ satisfying
  Condition (\ref{cond:1}) is maximal. And that, if the minimal solution satisfies (\ref{cond:1}), then
  the solution to (\ref{eq-bw2}) is unique.
  \item 
This result could be seen as an extension of Iyanaga's uniqueness theorem, where the 
function $a(t)$ was assumed to be continuous. For details see \cite[Theorem 1.13.1]{MR1336820}.
  \item 
   We could replace H3-H5 by the existence of
  a measurable non-negative function $g$ such that inequality 
  (\ref{eq-des1}) holds with $g(z)$ in place of $b'(z\rlim)$, and write $g(|y_t+\omega_t|)$
  in place of $a(t)$ in (\ref{cond:1}).
\end{enumerate}

\end{remarks}
\begin{proof} 
Set $\phi={\bar y}-y$ and $k(t,z)=a(t)z$, for $t\in[0,T]$ and $z\in\Reals$.
 Suppose that there is $t_1\in(0,T)$ such that $\phi_{t_1}=z_1>0$ and consider the function $u$
 such that
$$\begin{cases} u(t)=z_0+\int_0^tk(s,u(s))ds\ ,\quad t\in[0,T] \\
u(t_1)=z_1  \ .
\end{cases}$$

Note that the definition of $k$ yields that $z_0>0$.
 Hence, from Lemma \ref{lem-des-iyana}, we have
\begin{equation*}
  \phi_t=\bar y_t-y_t=\int_0^t\big[b(\bar y_s+\omega_s)-b(y_s+\omega_s)\big]\,ds
  \le \int_0^tk(s,\phi_s)ds\ ,\quad t\in[0,T]\ ,
\end{equation*}
and therefore
 $$\phi_t-u(t)\le -z_0+\int_0^t k(s,\phi_s-u(s))\,ds\ ,\quad t\in[0,T]\ .$$
 Thus, Lemma \ref{Pachpatte} 
 applied to $f(t):=\phi_t-u(t)$, $g(t):=-z_0\exp\{\int_0^t a(s)\,ds\}$ and $h(t):=-z_0$, leads to write
 $$\phi_t-u(t)\le -z_0\exp\Big\{\int_0^t a(s)ds\Big\}\ ,\quad t\in[0,T]\ ,$$
 which, for $t=t_1$,  gives $0\le-z_0\exp\big\{\int_0^{t_1} a(s)ds\big\}<0$, 
 a contradiction. Consequently, $\phi\equiv 0$
 on $[0,T]$ and 
 the proof is complete.
\end{proof}

We use Theorem 
\ref{the:3.2n} in the proof of Theorem  \ref{thm:main} below in order to show the path-by-path uniqueness  
for equation (\ref{eq-bw2}). 
Davie \cite{MR2377011} computes a difficult estimate of the moments of 
the integral
\begin{equation*}
\int_0^T \big(b(W_t+x)-b(W_t)\big)\,dt\ .
\end{equation*}
to replace Lipschitz-type
conditions in the study of a multidimensional version of (\ref{eq-bw2}) with bounded $b$.
The use of direct results on ordinary differential 
equations allows a different and shorter proof in dimension one, valid for unbounded coefficients $b$. In the proof, besides Theorem \ref{the:3.2n}, we make use of the following 
comparison theorem (see, for instance, Hartman \cite[Theorem III.4.1]{MR0171038}). 

\begin{lemma}\label{le-compthm}
  Let $h\colon[0,T]\times\Reals\rightarrow\Reals$ be a continuous function, $c\in\Reals$ and $u$ the minimal 
  solution to
  \begin{equation*}
    u'=h(t,u),\ \quad u_0=c\ .     
  \end{equation*}
  Also let $v$ be a differentiable function such that $v_0\ge c$ and $v'_t\ge h(t,v_t), \ t\in[0,T]$. Then, on the 
  interval of existence of $u$,  $v_t\ge u_t$. 
\end{lemma}

The following  is the main result of this section.
\begin{theorem}\label{thm:main}
 Let $W$ be a Brownian Motion on some probability space $(\Omega,{\cal F}, P)$, and 
 let $b\colon \Reals\rightarrow\Reals$ be a function satisfying hypotheses
 \emph{H1-H5} and: 
 \begin{enumerate} 
   \item  [\emph{H6.}]
   $b$ is continuous and $|b(x)|\le C(1+|x|)$, $\forall x$. 
   \item  [\emph{H7.}]
   $\displaystyle\E_P\Big[\int_0^T b'(|W_s|\rlim)\,ds \Big]<\infty$.
 \end{enumerate}
 Then, the stochastic differential equation
\begin{equation}\label{eq-W_t}
  X_t=\int_0^t b(X_s) \, ds + W_t
  \ ,\quad t\in[0,T] \ ,
\end{equation}
  has a unique path-by-path solution.  
\end{theorem}
\begin{proof}

Notice that, since $b$ is continuous, there exist minimal and maximal solutions to Equation
(\ref{eq-W_t}) for every continuous path $\omega$ of the Brownian motion $W$. First, we are going to 
construct a solution adapted to the natural filtration $\{{\cal F}_t\}_{t\in[0,T]}$  of $W$ which coincides with the minimal
solution to (\ref{eq-W_t}); secondly, we will see that this adapted solution also coincides with
the maximal solution. In conclusion, we will get the path-by-path uniqueness of the given 
stochastic differential equation. 
  
  Let $b_n(x):=b(x)-\frac{1}{n}$, $n\ge 1$. Consider a polynomial $p_n(x)$ such that 
\begin{equation*}
  |b_n(x)-p_n(x)|<\varepsilon_n
  \ ,\quad
  \text{for $x\in[-n,n]$}
  \ ,
\end{equation*}
  with $\varepsilon_n=\frac{1}{2n(n+1)}$, and extend it as $p_n(x)\equiv p_n(n)$, for $x\ge n$, and 
  $p_n(x)\equiv p_n(-n)$, for $x\le -n$. 
  
  Let $f(t,y):=b(y+\omega_t)$, and $f_n(t,y):=p_n(y+\omega_t)$.
   The functions $f_n$ are bounded, continuous and globally Lipschitz in the second variable, uniformly
  in the first. Therefore, 
  the stochastic differential equation $Y_t^n=\int_0^t f_n(s, Y_s^n)\, ds$ (equivalently,
  $X_t^n=\int_0^t p_n(X_s^n)\, ds + W_t$) has a unique
  $\{{\cal F}_t\}$-adapted solution, which is a path-by-path solution of the corresponding deterministic
  equation for almost all Brownian sample paths. 
  Also, $-2\le f_n\le f$ and $f_n$ converges to $f$ pointwise and monotonically from below.
  
  By Lemma \ref{le-compthm}, applied to
  \begin{equation*}
    \left\{\begin{array}{lll}
             {y_t^n}'=f_n(t,y_t^n) \\
             y_t'=f(t,y_t)\ge f_n(t,y_t) \\
             y_0^n=y_0=0 
             \ ,
           \end{array}
    \right.       
  \end{equation*}
  where $y$ is any solution of (\ref{eq-bw2}), we get $y\ge y^n$, on $[0,T]$.
  By the same comparison argument, since $\{f_n\}_n$ is non-decreasing, 
  the sequence of solutions $y^n$ is 
  non-decreasing as $n\to \infty$. 

  Clearly, 
  there exists a compact set $K\subset\Reals$ such that $y^n\colon[0,T]\rightarrow K$, for all $n$.
  Hence, by Dini's theorem, the sequence $f_n$ 
  converges uniformly to $f$ when all functions are considered on $[0,T]\times K$.

  We can then apply Theorem 1.2.4
  of Hartman \cite{MR0171038}, which states that
  a certain subsequence $y^{n_k}$ is uniformly convergent on $[0,T]$ to a solution
  of $y'=f(t,y)$. But since $\{y^n\}$ is increasing, it must itself converge to that solution. 
  Finally, given that $y^n$ is bounded from above by any solution of (\ref{eq-bw2}), the limit must be the
  minimal solution.
 
  The stochastic process $Y_t$ constructed in this way is therefore a solution,  
  $\{{\cal F}_t\}$-adapted, of the 
  stochastic differential equation  
   $Y_t=\int_0^t b(Y_s+W_s)\, ds$. Hence, the process $X_t:=Y_t+W_t$ is
  a strong solution of (\ref{eq-W_t}).

\bigskip

   For the second part of the proof, we start with the process $X$ just constructed, and prove first that
   $\int_0^T b'(|X_s|\rlim)\,ds$ is finite almost surely. Indeed,
   let $\bar X$ be a Brownian motion under some other probability $Q$ on $(\Omega,{\cal F})$; 
   thanks to the linear growth condition H6
   and \cite[Corollary 3.5.16]{MR1121940},
   Girsanov's theorem can be applied and there exists a probability $\bar P$ equivalent to $Q$
   such that $\bar X_t-\int_0^tb(\bar X_s)\,ds=:\bar W_t$ is a $\bar P$-Brownian motion. That means that 
   $(\bar X, \bar W)$ is a weak solution of equation (\ref{eq-W_t}). 
   
   By H7, and the equivalence of $\bar P$ and $Q$, we obtain
   
\begin{equation}\label{eq:barPbarX}
  \bar P\Big(\int_0^T b'(|\bar X_s|\rlim)\,ds <\infty\Big)=1
  \ .
\end{equation}
   The processes $X$ and $\bar X$ have a.s.~continuous paths under probabilities $P$ and $Q$ 
   (hence $\bar P$), respectively.  
    Therefore, 
   $P\big(\int_0^T b(X_s)^2\,ds<\infty\big)=1$ 
   and 
   $\bar P\big(\int_0^T b(\bar X_s)^2\,ds<\infty\big)=1$, thanks to H6.    
   Applying \cite[Proposition 5.3.10]{MR1121940}, we obtain that the laws of the vector processes
   $(X,W)$ under $P$ and
   $(\bar X,\bar W)$ under $\bar P$ are the same. 

  Consider the space of continuous functions $C([0,T])$ with its Borel $\sigma$-field and the $\Reals$-valued functional
  on $C([0,T])$ given by $\Lambda_g(x):=\int_0^T g(|x_s|)\,ds$, with $g\colon\Reals^+\rightarrow\Reals$ 
  continuous. $\Lambda_g$ is a continuous functional, and therefore measurable. 
  This is also true when $g$ is the indicator function of an interval, due to the dominated convergence 
  theorem. By the usual monotone class argument, 
  we get that $\Lambda_g$ is measurable for all bounded measurable functions $g$. And by the monotone convergence
  of the integrals, we get the same also for unbounded non-negative functions. That means that the law of the random variable $\Lambda_g(X)$ 
  under $P$ coincides with that of $\Lambda_g(\bar X)$ under $\bar P$, and in particular, applied to
  $g(z):=b'(|z|\rlim)$,  
\begin{equation*}
  P\Big(\int_0^T b'(|X_s|\rlim)\,ds<\infty\Big)=\bar P\Big(\int_0^T b'(|\bar X_s|\rlim)\,ds<\infty\Big)
  \ ,
\end{equation*}
  which together with (\ref{eq:barPbarX}) yields that $\int_0^T b'(|X_s|\rlim)\,ds$ is a.s.~finite, 
  as we wished to see. It means that
  the process $Y_t=X_t-W_t$ satisfies that 
  $t\mapsto b'(|Y_t+W_t|\rlim)$ is a.s.~integrable on $[0,T]$.
  Moreover, using hypotheses H3 and H7, the almost sure integrability of 
\begin{equation*}  
    \ind_{\{\omega_t\ge 0\}}
    \cdot
    b'\Big(\Big(\omega_t+\int_0^t\ind_{\{\omega_s>0\}}b(\omega_s)\,ds\Big)\rlim\Big)  
\end{equation*}
is immediate. Applying  Theorem   \ref{the:3.2n} 
  one concludes that $Y_t(\omega)$  
  coincides a.s.~with the maximal solution to the deterministic equation (\ref{eq-bw2}). 

We have seen that $Y_t(\omega)$ is both the minimal and maximal solution of (\ref{eq-bw2}). 
Hence, $X_t=Y_t+W_t$ is the unique path-by-path solution to the stochastic differential equation (\ref{eq-W_t}).
\end{proof}

\begin{remark}
  As a by-product of the proof, we have seen that under the conditions of the theorem, the unique path-by-path
  solution is also a strong solution, i.e.\ it is $\{{\cal F}_t\}$-adapted.
  Also, Theorem \ref{thm:main} is valid replacing $b'$ by a function $g$ 
   satisfying the conditions of Remark \ref{rem:g}(3) and hypothesis H7. 
\end{remark}

\begin{examples}\label{ex:sqrt}\emph{
    The paradigmatic function satisfying all our hypothesis is the square root: $b(x)=\sqrt{|x|}$.
    In fact, all functions of the form $b(x)=|x|^{\alpha}$, with $0<\alpha<1$, are continuous 
    non-Lipschitz functions satisfying H1-H7. To check H7, just notice that $W_s$ is Gaussian
    with variance $s$, and therefore 
\begin{equation*}
  \E[|W_s|^{\alpha-1}]=C_\alpha\cdot s^{\frac{\alpha-1}{2}}
  \ ,
  \quad 
  \text{for some constant $C_\alpha$, for every $\alpha>0$\ .}
\end{equation*}
}%
\emph{%
    The hypothesis of continuity of $b$ is needed in the proof of Theorem \ref{thm:main} 
    to guarantee the 
    existence of the uniform approximations of $b(x)-\frac{1}{n}$ on $[-n,n]$ by polynomials. 
    However, one can 
    allow for some discontinuities and ideas similar to those of the
    preceding proof can be applied.  
    For example, this is the true with the function 
    \begin{equation*}
      b(x)=\begin{cases}
             \sqrt{x}, & \text{if $x\ge 0$} 
             \\
             \sqrt{-x}+1, & \text{if $x<0$} \ .
           \end{cases}
    \end{equation*}
  We develop this particular case in the next section.
}
\qed
\end{examples}
Observe that the condition (\ref{cond:1}) for the uniqueness of solutions does not depend 
only on the noise function $\omega$ and the coefficient $b$, but also on the minimal 
solution $y$ to (\ref{eq-bw2}).
It is necessary to have an estimate of the type $b'(|y_t+\omega_t|\rlim)\le F(t)$,
with an integrable function $F$, to obtain the uniqueness from Theorem \ref{the:3.2n}.
Hypothesis H6 was only used to this purpose.
For a non-negative noise however, Condition (\ref{cond:1}) becomes 
\begin{equation}\label{cond:2}
 b'\Big(\Big(\omega_t+\int_0^t b(\omega_s)\,ds\Big)\rlim\Big)\in L^1([0,T])
\end{equation}
and such an estimate is not necessary. Hypotheses H4 and H5 are not needed either. 
For instance, we can prove the following result,
where the noise is the absolute value of a Brownian motion.

\begin{proposition}\label{prop-|W_t|}
  Let $W_t$ be a Brownian Motion, and consider the stochastic differential equation
\begin{equation}\label{eq-|W_t|}
  X_t=\int_0^t b(X_s) \, ds + |W_t|\ ,\quad t\in[0,T]\ .
\end{equation}
  Assume $b$ satisfies hypotheses H1-H3 and H7. Then, 
  equation (\ref{eq-|W_t|}) has at most one non-negative path-by-path solution, for almost all sample paths of $W$. 

  If, moreover, $b\ge 0$ on an interval $(-\varepsilon,0)$, then this is the unique path-by-path
  solution. 
\end{proposition}
\begin{proof}
  As we have already pointed out, we only need to show that condition (\ref{cond:2}) holds true 
  for almost all sample paths of $|W|$.  
  But this is an easy consequence of the facts that $b(|W|)$ is non-negative, $b'$ is non-increasing 
  on $(0,\infty)$
  and the expectation in H7 is finite. 
  If $b$ is also non-negative on an small interval to the left of 0, then any solution will be non-negative,
  and we get the path-by-path uniqueness. 
\end{proof}

\begin{remark}
One can also deduce from Condition (\ref{cond:2}) a result for negative noise: If $\omega\le 0$,
equation (\ref{eq-bw}) is equivalent to 
\begin{equation*}
  z_t=\int_0^t \tilde b(z_s)\,ds+\tilde\omega_t
\end{equation*}
where $\tilde b(z):=-b(-z)$ and $\tilde\omega_t:=-\omega_t$. Therefore we obtain in this case the uniqueness 
of a non-positive solution under condition  (\ref{cond:2}) and the hypotheses
\begin{enumerate}[H1.]
  \item
  $b(0)=0$. 
  \item [H2'.]
  $b$ non-decreasing on $(-\infty,0)$. 
  \item [H3'.]
  $b$ continuous on $(-\infty,0]$, and of class $C^1$ with $b'$ non-decreasing on $(-\infty,0)$.  
\end{enumerate}
  Hence, in this situation, the stochastic differential equation (\ref{eq-|W_t|}) with $-|W_t|$ 
  instead of $+|W_t|$ has a unique path-by-path non-positive solution; and it is the unique path-by-path
  solution if moreover $b\le 0$ on some interval $(0,\varepsilon)$.
\end{remark}

  Some known results for uniqueness in the theory of ordinary differential equations can be
  used in particular cases to obtain results similar to those above; this will be illustrated in the
  following sections. In Section \ref{sec-pos_noise} we consider 
  the discontinuous case based in the square root that was mentioned in Examples \ref{ex:sqrt},
  and the square root itself,  $b(x)=|x|^{1/2}$, for a non-negative disturbance.
  For the latter, the results are not really better
  than applying the general setting above, but they are easier to obtain by other means. In Section 
  \ref{sec:wdif}, we study the uniqueness of the 
  solution to equation (\ref{eq-bw2}) for some differentiable noises.
\section{The particular case of the square root}\label{sec-pos_noise}

\subsection{Example: square root with a discontinuity}
One can allow the function $b$ to have some discontinuities and still get uniqueness of solutions.
We illustrate this point with
    \begin{equation*}
      b(x)=\begin{cases}
             \sqrt{x}, & \text{if $x\ge 0$} 
             \\
             \sqrt{-x}+1, & \text{if $x<0$}\ .
           \end{cases}
    \end{equation*}
    Defining
    \begin{equation*}
      b_n(x)=\begin{cases}
               \sqrt{x}-\frac{1}{n}\ , & \text{if  $x> 0$}
               \\
               \sqrt{-x}+1-\frac{1}{n}\ , & \text{if $x<-\frac{1}{n}$}
               \\
               -(n+\sqrt{n})x -\frac{1}{n}\ , & \text{if $-\frac{1}{n}\le x \le 0$\ ,}
             \end{cases}
    \end{equation*}
     
we have that $\{b_n:\ n\in\mathbb{N}\}$ is a sequence of continuous functions on $\mathbb{R}$
such that, for $x\in\mathbb{R}$,
\begin{equation}\label{eq:bybn}
b_n(x)\le b(x)\quad\mbox{\rm and}\quad b_{n+1}(x)-b_n(x)\ge \frac{1}{n(n+1)}\ .
\end{equation}
As in the proof of Theorem \ref{thm:main} we consider a polynomial $p_n$ such that
\begin{equation*}
  |b_n(x)-p_n(x)|<\varepsilon_n
  \ ,\quad
  \text{for $x\in[-n,n]$}
  \ ,
\end{equation*}
  with $\varepsilon_n=\frac{1}{2n(n+1)}$, and extend it as $p_n(x)\equiv p_n(n)$, for $x\ge n$, and 
  $p_n(x)\equiv p_n(-n)$, for $x\le -n$. 
The definitions of $b_n$ and $p_n$, together with (\ref{eq:bybn}), allow us to deduce that, for $x\in\mathbb{R}$,
\begin{equation}\label{eq:bnypn}
p_{n+1}(x)-p_n(x)\ge 0 \quad\mbox{\rm and}\quad b(x)\ge p_n(x)+\varepsilon_n \ .
\end{equation}
Hence, $-\varepsilon_n-\frac{1}{n}\le p_n(x)<b(x)$, $x\in\mathbb{R}$. Therefore,
using that $b$ has linear growth, we can find a constant $K>0$ such that
\begin{equation}\label{eq:lgpn}
\left|p_n(x)\right|\le K(1+|x|)\ ,\quad \hbox{for }\ x\in\mathbb{R}\ \ \hbox{\rm and }
\ n\in\mathbb{N} \ .
\end{equation}

Now we consider 
\begin{equation}\label{eq:spm}
Y_t^{(m)}=-{\tilde\varepsilon}_m+\int_0^tp_m(Y_s^{(m)}+W_s)\,ds\ ,\quad t\in[0,T]\ ,
\end{equation}
where ${\tilde\varepsilon}_m\downarrow 0$ as $m\rightarrow\infty$, and $W$ is a Brownian motion.
Observe that the fact that $p_m$ is a bounded Lipschitz function implies that equation (\ref{eq:spm})
has a unique solution, which is measurable on $\Omega\times[0,T]$ and 
adapted with respect to the filtration $\{\mathcal{F}_t\}$ generated by $W$.
In order to see that the minimal solution to equation (\ref{eq-W_t}) is also
measurable  and $\{\mathcal{F}_t\}$-adapted, we establish the following lemma.
\begin{lemma}\label{lem:com-eqs}
Let $Y$ be a solution of equation 
\begin{equation}\label{eq:stoch_y}
  Y_t=\int_0^t b(Y_s+W_s)\,ds\ ,\quad t\in [0,T]\ ,
\end{equation}
$m\in\mathbb{N}$ and 
$Y^{(m)}$ the solution of (\ref{eq:spm}). Then, 
$$Y_t\ge Y_t^{(m+1)}\ge Y_t^{(m)}\ , \quad t\in[0,T]\ .$$
\end{lemma}
\begin{proof}
By Lemma \ref{le-compthm} and (\ref{eq:bnypn}), we only need to see that $Y_t\ge Y_t^{(m)}$, for $t\in[0,T]$.
By the continuity of $Y$ and $Y^{(m)}$,  and ${\tilde\varepsilon}_m>0$, there is  $t_0\in(0,T]$
such that $Y_t>Y^{(m)}_t$, for $t\in[0,t_0]$. Now suppose that there exist $t_1<T$
and $\eta>0$ such that
\begin{equation*}
Y_{t_1}=Y^{(m)}_{t_1}\quad\hbox{\rm and}\quad Y_{t}^{(m)}>Y_{t}\ ,\quad\text{for $t\in[t_1,t_1+\eta]$} \ .
\end{equation*}
Then, for $h>0$ small enough,
$$\frac{Y^{(m)}_{t_1+h}-Y^{(m)}_{t_1}}{h}=\frac{Y^{(m)}_{t_1+h}-Y_{t_1}}{h}
>\frac{Y_{t_1+h}-Y_{t_1}}{h}\ .$$
Consequently, 
\begin{equation*}
p_m(Y_{t_1}+W_{t_1})=p_m(Y_{t_1}^{(m)}+W_{t_1})\ge D^+Y_{t_1}\ ,
\end{equation*}
with $D^+Y_{t_1}=\limsup_{h\downarrow 0}\frac{Y_{t_1+h}-Y_{t_1}}{h}.$

On the other hand, (\ref{eq:bnypn}) leads to write
$$\frac{Y_{t_1+h}-Y_{t_1}}{h}=\frac{1}{h}\int_{t_1}^{t_1+h}b(Y_s+W_s)\,ds
>\frac{1}{h}\int_{t_1}^{t_1+h}\left(p_m(Y_s+W_s)+\varepsilon_m\right)ds\ .
$$
Therefore,
$$D^+Y_{t_1}\ge p_m(Y_{t_1}+W_{t_1})+\varepsilon_m>p_m(Y_{t_1}+W_{t_1})\ ,$$
a contradiction.
The proof is complete.
\end{proof}

Now we introduce the measurable and $\{\mathcal{F}_t\}$-adapted process 
${\bar Y}_t:=\lim_{m\rightarrow\infty}Y^{(m)}_t$, which is 
well-defined due to Lemma \ref{lem:com-eqs}.
\begin{lemma}\label{lem:byabsc}
The process ${\bar Y}$ is absolutely continuous (i.e.\ it has absolutely continuous paths).
\end{lemma}
\begin{proof}
Let $W^{*}_T=\sup_{t\in[0,T]}|W_t|$. Then (\ref{eq:lgpn}) yields, for some constant $K$, 
$$ |Y^{(n)}_t|\le K\int_0^T|Y^{(n)}_s|\,ds+K(1+T+TW^{*}_T)\ ,\quad t\in[0,T]\ .$$
Thus, Gronwall's lemma  implies
\begin{equation}\label{eq:cgl}
|Y^{(n)}_t|\le K(1+T+TW^{*}_T)\exp(KT)\ ,\quad t\in[0,T]\ .
\end{equation}
It therefore follows that there exists a positive constant $C$ such that, for
$0\le t_1<t_2<\ldots<t_{\ell}\le T$,
\begin{equation}\label{eq:abscontbound}
\sum_{i=1}^{\ell-1}|Y^{(n)}_{t_{i+1}}-Y^{(n)}_{t_i}|\le
K\sum_{i=1}^{\ell-1}\int_{t_i}^{t_{i+1}}\left(1+|Y^{(n)}_s+W_s|\right)ds
\le C(1+W^{*}_T)\sum_{i=1}^{\ell-1}(t_{i+1}-t_i)\ .
\end{equation}
Finally, we prove the assertion of the lemma by letting $n\rightarrow\infty$.
\end{proof}

Observe that an immediate consequence of Lemma \ref{lem:byabsc} and (\ref{eq:abscontbound})
(with $\ell=2$)
is that there exists
a measurable and $\{\mathcal{F}_t\}$-adapted process $A$ such that
\begin{equation}\label{eq:cotA}
{\bar Y}_t=\int_0^t A_s\,ds\quad \hbox{\rm and}\quad |A_t|\le C(1+W^{*}_T)\ ,
\quad t\in[0,T]\ .
\end{equation}

Now we can state the main result of this example.
\begin{theorem}
The process ${\bar Y}$ is the unique path-by-path solution of equation (\ref{eq:stoch_y}).
\end{theorem}
\begin{proof}
We first observe that (\ref{eq:cotA}) and Girsanov's theorem (see Theorem 3.5.1 and
Corollary 3.5.16 in Karatzas and Shreve \cite{MR1121940}) imply that
$W_t+{\bar Y}_t\neq 0$ for almost all $t\in[0,T]$, with probability 1. Now choose $s\in[0,T]$ so that
$W_s+{\bar Y}_s\neq 0$ a.s. Then
\begin{align*}\lefteqn{
\left|p_n(Y^{(n)}_s+W_s)-b({\bar Y}_s+W_s)\right|}\\
&\le\left|p_n(Y^{(n)}_s+W_s)-b_n(Y^{(n)}_s+W_s)\right|+
\left|b_n(Y^{(n)}_s+W_s)-b(Y^{(n)}_s+W_s)\right|\\
&\phantom{\le}+
\left|b(Y^{(n)}_s+W_s)-b({\bar Y}_s+W_s)\right|\\
&\le \varepsilon_n+\frac{1}{n}+
\left|b(Y^{(n)}_s+W_s)-b({\bar Y}_s+W_s)\right|\ .
\end{align*}
So we can conclude that $p_n(Y^{(n)}_s+W_s)\rightarrow b({\bar Y}_s+W_s)$ a.s.~as
$n\rightarrow\infty$ due to the continuity of $b$ on $\mathbb{R}-\{0\}$. Hence,
(\ref{eq:lgpn}), (\ref{eq:spm})  and (\ref{eq:cgl})  give
$${\bar Y}_t=\int_0^tb({\bar Y}_s+W_s)ds\ ,\quad t\in[0,T]\ .$$

We have obtained that the $\{{\cal F}_t\}$-adapted process $\bar Y$ is, by Lemma \ref{lem:com-eqs}, 
the minimal solution, a.s. 
Now we can finish
as in the proof of Theorem \ref{thm:main}. Instead of the continuity of $b$ it is enough that $b$ be
locally bounded to use that $P\{\int_0^t b(X_s)^2\,ds<\infty\}=1$.
\end{proof}

\subsection{Example: Square root and non-negative noise}
We assume in this section that $\omega\colon[0,T]\rightarrow[0,\infty)$ is a fixed continuous 
non-negative function. Consider the equation 
\begin{equation}\label{eq-sqrtw} 
   x_t=\int_0^t \sqrt{|x_s|}\,ds + \omega_t\ , \quad t\in[0,T]\ .
\end{equation}

and its equivalent, defining
$y_t=x_t-\omega_t$,
\begin{equation}\label{eq-sqrtw2} 
   y_t=\int_0^t \sqrt{|y_s+\omega_s|}\,ds\ , \quad t\in[0,T]\ .
\end{equation}
Any solution $y$ of (\ref{eq-sqrtw2}) is clearly a continuously differentiable, positive and non-decreasing function. The absolute 
value inside the square root is therefore unnecessary.

We will make use of the following uniqueness theorem (see Agarwal and Lakshmikantham \cite[Theorem 2.8.3]{MR1336820}): 
\begin{lemma}\label{LakshUT}
   (Lakshmikantham's Uniqueness Theorem). Suppose that    
$f(t,y)$ is defined in $D:=(0,T]×[-a,a]$, measurable in $t$ for each fixed
$y$, continuous in $y$ for each fixed
$t$, and there exists an integrable
function
$M$ on the interval $[0, T]$ such that 
$|f(t,y)|\le M(t)$ on $D$. Consider the ordinary one-dimensional differential
   equation 
\begin{equation}\label{eq-y'=f(t,y)}   
   y'(t)=f(t,y(t)) \ , \quad y(0)=0 \ ,
\end{equation}  
and define a \emph{classical solution} of (\ref{eq-y'=f(t,y)}) on $[0,T]$ as a function $y$ satisfying the initial condition, continuous in $[0,T]$,
and differentiable and verifying the equation on $(0,T]$.

Assume that:
\begin{enumerate}[i)]   
\item\label{LakshUT_i}
Any two classical solutions
$y$ and $\bar y$ of (\ref{eq-y'=f(t,y)})  satisfy 
\begin{equation*}
  \lim_{t\to 0+} \frac{|\bar y_t - y_t|}{B_t}=0
  \ ,
\end{equation*}
where 
$B$ is a continuous and positive function on $(0, T]$ with $\lim_{t\to 0^+} B(t) = 0$.
\item\label{LakshUT_ii} 
There is a continuous and non-negative function
$g\colon (0,T]\times[0,2a]$ for which 
the only solution $z$ of $z^{'}_t=g(t,z)$ on $[0, T]$ such that $\lim_{t\to0^+} \frac{z_t}{B_t} = 0$ 
is the trivial solution $z\equiv 0$.
\item\label{LakshUT_iii} 
$f$ is defined on $\bar D$ (the closure of $D$), and for all $(t, y)$ and $(t, \bar y)$ in 
$D$, the inequality $|f(t,\bar y)-f(t,y)|\le g(t,|\bar y - y|)$ is satisfied.
\end{enumerate}
  Then, equation (\ref{eq-y'=f(t,y)})  above has at most one classical solution on $[0,T]$. 
\end{lemma}

Our equation reads $y'_t=f(t,y_t)$, with $f(t,y)=\sqrt{y+\omega_t}$. In this case, moreover, the function $f$ is continuous, and Lakshmikantham
theorem implies the uniqueness of ordinary solutions (i.e.\ of class $C^1$ in $[0,T]$).

Let $y$ and $\bar y$ be the minimal and maximal solutions of (\ref{eq-sqrtw2}), respectively. 
By the mean value theorem applied to 
$f(x):=\sqrt{x}$, we have 
\begin{equation*}
  \sqrt{\bar y_t+\omega_t}-\sqrt{y_t+\omega_t} = 
  \frac{\bar y_t-y_t}{2}\cdot(\omega_t+\xi_t)^{-1/2}\ ,\quad t\in[0,T]\ ,
\end{equation*}
  for some $\xi_t\in[y_t, \bar y_t]$. Since $\xi_t\ge y_t\ge \int_0^t \omega^{1/2}_s\,ds$, we find the bound
\begin{equation}\label{eq-bound_f(t,y)}
  \sqrt{\bar y_t+\omega_t}-\sqrt{y_t+\omega_t} \le \frac{\bar y_t-y_t}{2}\cdot\Big(\omega_t+
  \int_0^t \omega^{1/2}_s\,ds\Big)^{-1/2}\ , \quad t\in[0,T]\ .
  \end{equation}

We also have that (\ref{eq-sqrtw2}) implies
\begin{equation*}
(\bar y_t-y_t)' \le \sqrt{\bar y_t-y_t}\ ,\quad t\in[0,T]\ .
\end{equation*}
Using 
Lakshmikantham and Leela \cite[Theorem 1.4.1]{MR0379933}, the difference $\bar y_t-y_t$ is bounded
by the maximal solution to $z_t=\int_0^t\sqrt{z_s}\,ds$, which is $z_t=t^2/4$.
Now, taking $B(t)=t^{\alpha}$, with any $\alpha\in(0,2)$, 
 hypothesis (\ref{LakshUT_i}) of Lemma \ref{LakshUT} is clearly satisfied.
  
  For conditions (\ref{LakshUT_ii}) and (\ref{LakshUT_iii}), notice that, by (\ref{eq-bound_f(t,y)}), we can take 
\begin{equation*}
g(t,z):=
  \frac{z}{2}\cdot\Big(\omega_t+\int_0^t \omega^{1/2}_s\,ds\Big)^{-1/2},\quad
  t\in(0,T]\ \hbox{\rm and}\ z\in\Reals^+\ , 
\end{equation*}
 assuming the expression on the right makes sense.
 The differential equation
  $z'_t=g(t,z_t)$ is linear, and all its solutions can be explicitly written as 
\begin{equation*} 
  z_t= c\exp\Big\{\frac12 \int_{t_0}^t\Big(\omega_s+\int_0^s \omega^{1/2}_rdr\Big)^{-1/2}\,ds\Big\}
\end{equation*}
  for some constant $c$ and $t_0\in(0,T]$. Then, if $s\mapsto \omega_s+\int_0^s \omega_r\,dr$ 
  is integrable at $0^+$, those solutions can only tend to zero at the origin if $z\equiv 0$. 

 Thus, we have proved the following result.
\begin{theorem}\label{the-nuevo}
Assume that the noise $\omega$ is such that
 $\Big(\omega_\cdot+\int_0^\cdot \omega^{1/2}_s\,ds\Big)^{-1/2}\in L^1([0,T])$.
Then, there exists a unique solution to equation (\ref{eq-sqrtw}).
\end{theorem} 

As an immediate consequence of this theorem, we recover the result of Proposition \ref{prop-|W_t|} in an easier way:
\begin{corollary}
  Let $W$ be a Brownian Motion, and consider the stochastic differential equation 
\begin{equation}\label{eq-|W_t|2}
  X_t=\int_0^t \sqrt{|X_s|} \, ds + |W_t|\ ,\quad t\in[0,T]\ .
\end{equation}
  Then, equation (\ref{eq-|W_t|2}) has a unique path-by-path solution, for almost all paths of $W$.
\end{corollary}
\begin{proof}
  In view of Theorem \ref{the-nuevo}, we only have to show that
  for almost all sample paths $\omega$ of a Brownian motion, 
\begin{equation*}
\Big(|\omega_\cdot|+\int_0^\cdot |\omega_s|^{1/2}\,ds\Big)^{-1/2}\in L^1([0,T])\ ,
\end{equation*}
and this has already been checked in Examples \ref{ex:sqrt}.

%
\end{proof}

\section{Differentiable noise}\label{sec:wdif}
In this section we analyze the uniqueness of a solution to equation (\ref{eq-bw}) for some differentiable  
perturbations. Equivalently, we are dealing with the absolutely continuous solutions to 
\begin{equation}\label{eq:diffnoise}
  \begin{cases}
    x'_t=b(x_t)+w'_t\ ,\quad \text{$t$-a.e. on $[0,T]$}
    \\
    x_0=0 
    \ ,
  \end{cases}
\end{equation}
 where $\omega$ is a function in $C^1([0,T])$, and we want to keep at a minimum the hypotheses on $b$. 

We state first a general result for noises with a strictly negative derivative.
We mimic the proof of Peano's uniqueness theorem (see, for instance, 
\cite[Theorem 1.3.1]{MR1336820}). 

\begin{theorem}\label{the:nne}
Let $\omega$ be a $C^1$ function on $[0,T]$ such that $\omega_0=0$ and with negative derivative
(i.e.\ $\omega^{'}_t<0$ for $t\in[0,T]$). 

Assume that 
\begin{enumerate}[i)]
  \item 
  $b$ is measurable and $\lim_{x\to 0} b(x)=b(0)=0$.
  \item
  For some $\eta>0$, there exists an increasing continuous function $g\colon[0,\eta)\rightarrow\Reals$, 
  of class $C^1$ on $(0,\eta)$, with $g'$ non-increasing, and such that:
\begin{enumerate}
  \item [\emph{H8.}]
  $x\mapsto g'(-x)b(x)$ is non-increasing on $(-\eta,0)$.
\end{enumerate}
  \item
  There exists either a maximal or a minimal solution to (\ref{eq:diffnoise}).
\end{enumerate}

Then, there is a unique local solution to equation (\ref{eq:diffnoise}). Global uniqueness on $[0,T]$
is true if $\eta=\infty$. 
\end{theorem}

\begin{proof}
Let $x$ and ${\bar x}$ be two solutions such that $x\le {\bar x}$. 
We first observe that for some $\varepsilon>0$, we have $-\eta/2\le x,{\bar x}<0$ on $(0,\varepsilon)$
due to the continuity of $\omega'$ and $b$ at zero, and to $\omega'_0<0$.

Define $z_t:=g(-x_t)$ and ${\bar z}_t:=g(-{\bar x_t})$ on $[0,\varepsilon)$. Both $z$ and $\bar z$
are absolutely continuous since $g$ is $C^1$ and $x$ and $\bar x$ are absolutely continuous,
and 
\begin{equation*}
  z'_t=-g'(-x_t)(b(x_t)+\omega'_t)  
   \quad \text{and} \quad 
  \bar z'_t=-g'(-\bar x_t)(b(\bar x_t)+\omega'_t)
  \ ,\quad
  \text{$t$-a.e. on $(0,\varepsilon)$}
  \ .
\end{equation*}
The fundamental theorem of calculus gives, for $0<\delta<t<\varepsilon$,
\begin{align*}
  (z_t-\bar z_t)-(z_\delta-\bar z_\delta)
  &=\int_\delta^t \big[-g'(-x_s)b(x_s)+g'(-\bar x_s)b(\bar x_s)\big]\,ds
  +
  \int_\delta^t -\omega'_s\cdot\big(g'(-x_s)-g'(-\bar x_s)\big)\,ds
  \\ 
  &\le
  \int_\delta^t -\omega'_s\cdot\big(g'(-x_s)-g'(-\bar x_s)\big)\,ds 
  \ ,
\end{align*}
since, by H8, the first integral is non-positive.
Letting $\delta\to 0$ and using \emph{ii)}, we obtain $0\le z_t-\bar z_t\le 0$. 
Consequently we have that $z={\bar z}$ on $[0,\varepsilon)$. And, since $g$ is increasing, 
we get $x=\bar x$ on $[0,\varepsilon]$.

Now assume that the function $g$ is defined on $[0,\infty)$, and that uniqueness holds up to 
$t_0<T$. Any solution $x$ will satisfy
\begin{equation*}
x_t=x_{t_0}+\int_{t_0}^t b(x_s)ds+\omega_t-\omega_{t_0}\ ,\quad t\in[t_0,T]\ ,
\end{equation*}
where $x_{t_0}$ is a common value to all of them. Notice that every time $x$ hits the origin, it is differentiable
at that point and its derivative is negative. Therefore, $x_{t_0}\le 0$ and, by continuity, two solutions 
$x$ and $\bar x$ will be negative in some interval $(t_0, t_0+\varepsilon)$. 
We can proceed again
as in the beginning of the proof to extend uniqueness beyond $t_0$.
\end{proof}

\begin{remarks}\hspace{1pt}
  \begin{enumerate}[1.] 
    \item 
    A well known sufficient condition for the existence of maximal and minimal solutions is 
    the continuity of $b$.
    But weaker conditions exists in the literature. For instance, in the situation given, this is true if:
    \begin{enumerate}
      \item $b$ has linear growth, and 
      \item $\limsup_{y\to x^-} b(y) \le b(x) \le \liminf_{y\to x^+} b(y)$, for all $x$.
    \end{enumerate}
    These conditions follow easily from the general Theorem 3.1 in \cite{MR1689280}. Even weaker conditions, 
    allowing jumps in the ``wrong direction'', can be found
    in \cite{MR1868338} and \cite{MR2525582}.
    \item
    The continuity of $b$ at zero can be replaced by other conditions ensuring that the solutions remain
    negative. For example, if 
    \begin{enumerate}
      \item $\limsup_{x\to 0} b(x)\le 0$, or 
      \item There is a non-decreasing continuous function $f$ such that $b\le f$ on an open interval containing 0 and
      $f(0)+\omega'_0<0$. \\
      In this case, 
      \begin{equation*}
        x_t=\int_0^t \big(b(x_s)+\omega'_s\big)\, ds
        \le
        \int_0^t \big(f(x_s)+\omega'_s\big)\, ds
      \end{equation*}
      and $x$ is bounded by the maximal solution of 
      \begin{equation*}
        u_t=\int_0^t \big(f(u_s)+\omega'_s\big)\, ds
      \end{equation*}
      (see Pachpatte \cite[Theorem 2.2.4]{MR1487077}), which is negative on an interval $(0,\eta)$.
    \end{enumerate}  
    \item An example where the above remarks apply is given by
    \begin{equation*}
      b(x)=\begin{cases}
             \sqrt{x}, & \text{if $x\ge 0$} 
             \\ 
             \sqrt{-x}-1, & \text{if $x<0$} \ .
           \end{cases}
    \end{equation*}
    A maximal solution exists by the sufficient conditions of statement 1. Then both (a) or (b) of statement 2
   are applicable with $f(x)=\sqrt{x}\cdot\ind_{\{x\ge 0\}}$ in the second case.
  \end{enumerate}  
\end{remarks}

The following result is also inspired in the proof of Peano's uniqueness theorem.
We consider a particular example of an ordinary differential equation driven by 
a differentiable
noise, positive in a neighbourhood of zero, but changing sign afterwards. By ``piecewise Lipschitz'' below
we mean a function whose domain can be partitioned into intervals such that their interior is non-empty and the
function is Lipschitz on each of them.
\begin{example}\label{ex-posnoise}
Consider $\omega_t=\alpha t+t^{2+\beta}\sin(t^{-1})$,
where $\alpha,\,\beta>0$, $t\in(0,T]$ and $\omega_0=0$. 

Assume:
\begin{enumerate}[i)]
  \item $b$ is measurable and $\lim_{x\to 0} b(x)=b(0)=0$.
  \item For some $\eta>0$, there exists an increasing continuous function 
  $h\colon [0,\eta)\rightarrow \Reals$, of class $C^1$ on $(0,\eta)$, with $h'$ non-increasing, 
and such that:
  \begin{enumerate}
    \item [\emph{H9.}]
    $x\mapsto h'(x)b(x)$ is non-increasing on $(0,\eta)$.
  \end{enumerate}
  \item There exists either a maximal or a minimal solution to (\ref{eq:diffnoise}).
\end{enumerate}

Then, there is a unique local solution to equation (\ref{eq:diffnoise}). Global uniqueness on $[0,T]$
is true if, furthermore, 
\begin{enumerate}[i')]
  \item $b$ is non-negative, either piecewise Lipschitz or locally Lipschitz on $(-\infty,0)$, and locally Lipschitz on $(0,\infty)$. 
  \item $\eta=\infty$.
  \item[iv)] 
    There exists $g\colon [0,\infty)\rightarrow \Reals$ satisfying assumption ii) of Theorem \ref{the:nne}, 
\end{enumerate}

Then, 
equation (\ref{eq:diffnoise}) has a unique solution.
\end{example}
\begin{proof}
If $x$ is a solution to (\ref{eq:diffnoise}) with the given noise $\omega$, we can see, 
as in the preceding theorem, that there is an $\varepsilon>0$ 
such that $0<x<\eta/2$ and $\omega'>0$ on $(0,\varepsilon)$. Given any two such solutions with 
$x\le \bar x$, define $z_t:=h(x_t)$ and $\bar z_t:=h(\bar x_t)$, on $[0,\varepsilon)$. 

Hence, proceeding as in the proof of Theorem \ref{the:nne} but using hypothesis H9 instead of H8, we obtain that
\begin{equation*}
  0\le \bar z_t-z_t\le \int_0^t \omega'_s\cdot \big(h'(\bar x_s)-h'(x_s)\big)\,ds \le 0\ , 
\end{equation*}
and therefore
equation (\ref{eq:diffnoise}) with the given $\omega$ has a unique solution on $[0,\varepsilon]$.

For the second part, assume that uniqueness holds up to $t_0<T$.
We distinguish the following cases:

\begin{description}
  \item [{\sc Case 1:}  \textmd{$x_{t_0}>0$}].

  We only need to use that $b$ is locally Lipschitz to extend the uniqueness to the right of $t_0$.
        
  \item [{\sc Case 2:}  \textmd{$x_{t_0}\le 0$, $\omega'_{t_0}<0$}].
    
  Here we use condition H8, and we finish as in Theorem \ref{the:nne}.

  \item [{\sc Case 3:}  \textmd{$x_{t_0}\le 0$, $\omega'_{t_0}\ge 0$}].    
   
We can write
\begin{align}\label{dobder}
\omega^{''}_{t_0}&=\frac{1+\beta}{t_0}\omega'_{t_0}
-(\beta+1)\frac{\alpha}{t_0}-(\beta+1)t_0^{\beta-1}\cos(t_0^{-1})
-t_0^{\beta-2}\sin(t_0^{-1})\nonumber \\
&\ge -(\beta+1)\frac{\alpha}{t_0}-(\beta+1)t_0^{\beta-1}\cos(t_0^{-1})
-t_0^{\beta-2}\sin(t_0^{-1})\nonumber \\
&\ge -(\beta+1)\frac{\alpha}{t_0}+(\beta+1)\left[\frac{-\alpha}{t_0}
-(\beta+2)t^{\beta}_0\sin(t_0^{-1})\right]-t_0^{\beta-2}\sin(t_0^{-1})\ .
\end{align}
On the other hand, the facts that $x_{t_0}\le 0$ and $b$ is non-negative imply 
$\omega_{t_0}\le 0$. Thus,
$-t_0^{\beta+2}\sin(t_0^{-1})\ge\alpha t_0>0$, which, together with
(\ref{dobder}), yields
\begin{align*}
\omega^{''}_{t_0}&>(\beta+1)\left[-2\frac{\alpha}{t_0}-(\beta+2)t^{\beta}_0
\sin(t_0^{-1})\right]\nonumber\\
&\ge -(\beta^2+\beta)t_0^{\beta}\sin(t_0^{-1})>0\ .
\end{align*}

Therefore, there exists $\varepsilon>0$ such that
$\omega^{'}>0$ on $(t_0,t_0+\varepsilon)$. If $x_{t_0}=0$, we can proceed as in 
the beginning of this proof;
if $x_{t_0}<0$, and since $b$ is non-negative,   
    we have $x'>0$ on 
   $(t_0,t_0+\varepsilon)$. Therefore $x_t$ is increasing and the 
   piecewise Lipschitz property of $b$ on $(-\infty,0)$ gives the uniqueness beyond
   $t_0$, even if $b$ is discontinuous at $t_0$.
\end{description}
\end{proof}

Hypotheses H8 and H9 in Theorem \ref{the:nne} and Example \ref{ex-posnoise} are satisfied by functions of the 
form 
\begin{equation*}
  b(x)=
  \begin{cases}
    r_1(x)\cdot s_1(x)\ ,\quad x\ge 0
    \\
    r_2(x)\cdot s_2(-x)\ ,\quad x<0 \ ,
  \end{cases}
\end{equation*}
where 
$r_1$ and $r_2$ are non-negative and non-increasing, with $r_2$ piecewise locally Lipschitz, 
and $s_1$ and $s_2$ are positive, non-decreasing, continuous on $[0,\infty)$, 
with $1/s_1$ and $1/s_2$ integrable at zero. One can take $h(x)=\int_0^x 1/s_1$, and 
$g(x)=\int_0^x 1/s_2$. In particular, this family includes the non-Lipschitz functions 
$b(x)=|x|^{\alpha}$ ($0<\alpha<1$),
and, more generally, $b(x)=r(x)\cdot|x|^{\alpha}$, with convenient $r$; it suffices to take 
$g(x)=h(x)=\frac{1}{1-\alpha}x^{1-\alpha}$.

\section{Acknowledgments}

This work was partially supported by grants numbers MTM2011-29064-C03-01  
                               from the Ministry of Economy and Competitiveness of Spain, and 
                               UNAB10-4E-378, co-funded by the European Regional Development Fund,
                               and by the CONACyT grant 220303. 
                               
Both authors are thankful for the hospitality and economical support of the                    
Departamento de Control Automático of CINVESTAV-IPN, Mexico City.

\bibliographystyle{plain}
\bibliography{UniquenessSingularSDE}
\end{document}